\documentclass[12pt]{amsart}
\usepackage{amssymb}
\usepackage{mathscinet}
\usepackage{hyperref}
\usepackage{xcolor}
\usepackage{cite}

\newtheorem{thm}{Theorem}[section]
\newtheorem{defn}[thm]{Definition}

\newtheorem{ex}[thm]{Example}
\newtheorem{prop}[thm]{Proposition}

\newtheorem{lemma}[thm]{Lemma}

\DeclareMathOperator{\Inv}{Inv}
\DeclareMathOperator{\Fix}{Fix}
\DeclareMathOperator{\tr}{tr}

\newcommand*{\os}[1]{[ #1 ]}
\newcommand*{\sC}{\mathsf{C}}
\newcommand*{\fsl}{\mathfrak{sl}}
\newcommand*{\ct}{\mathbf{ct}}

\newcommand*{\Dfn}[1]{\emph{\color{blue}#1}}
\newcommand*{\bQ}{\mathbb{Q}}
\newcommand*{\bZ}{\mathbb{Z}}
\newcommand*{\fC}{\mathfrak{C}}
\newcommand*{\fS}{\mathfrak{S}}

\title{Interpolating between promotion and the long cycle}
\author{Bruce W. Westbury}
\email{Bruce.Westbury@utdallas.edu}
\address{Mathematical Sciences \\
	University of Texas at Dallas \\ 800 W Campbell Rd. \\ Richardson \\ TX 75080 \\ USA}
\date{June 2019}

\begin{document}
\begin{abstract} We give a new proof of the cyclic sieving phenomena
for promotion on rectangular standard tableaux. This uses an action
of the cactus groups in the seminormal bases of the irreducible 
representations of the Hecke algebras.
\end{abstract}

\maketitle
\tableofcontents

\section{Introduction}
The cyclic sieving phenomenon for rectangular standard tableaux was first
proved in \cite{Rhoades2010}. This proof used the Khazdhan-Lusztig
basis of the Hecke algebras. This result has been reproved in \cite{Purbhoo2013} using the geometry of the Wronksian and has been generalised in using \cite{Fontaine2014} using the geometry
of the affine Grassmannian and in \cite{Westbury2016} using Lusztig's
based modules. Here we give a new proof which is relatively self-contained
using the action of the cactus groups in the seminormal bases for the
irreducible representations of the Hecke algebras.

Our main theorem has a straightforward statement.
Let $\lambda$ be a partition of size $r$ and let $U_\lambda$ be the
associated irreducible representation of the symmetric group, $\fS_r$,
over the rational field, $\bQ$. Let $c_1$ be the matrix representing the
long cycle with respect to a chosen basis.
Let $c_0$ be the permutation matrix of the jeu-de-taquin promotion acting
on the set of standard tableaux of shape $\lambda$.
Then our main theorem is:
\begin{thm}\label{thm:main} If the partition $\lambda$ has rectangular shape then the matrices $c_1$ and $c_0$ are conjugate.
\end{thm}

The method of proof is to construct an interpolating matrix. This is a matrix, $c_q$, with entries in the field of rational functions, $\bQ(q)$, with
the properties:
\begin{itemize}
	\item The evaluation of $c_q$ at $q=1$ is defined and gives $c_1$
	\item The evaluation of $c_q$ at $q=0$ is defined and gives $c_0$
	\item $c_q^r=1$
\end{itemize}

This proves Theorem~\ref{thm:main} since $c_q$ is semisimple, the eigenvalues
are $r$-th roots of unity and the eigenvalues are analytic functions of $q$.
Hence the eigenvalues and their multiplicities are independent of $q$.

\begin{ex}
The rotation matrix and its inverse are
\begin{equation*}
\left(\begin{array}{rrrrr}
\frac{1}{3} & \frac{4}{9} & 0 & \frac{2}{3} & 0 \\
\frac{1}{2} & -\frac{1}{12} & \frac{3}{8} & -\frac{1}{8} & \frac{9}{16} \\
1 & -\frac{1}{6} & -\frac{1}{4} & -\frac{1}{4} & -\frac{3}{8} \\
0 & \frac{1}{2} & \frac{3}{4} & -\frac{1}{4} & -\frac{3}{8} \\
0 & 1 & -\frac{1}{2} & -\frac{1}{2} & \frac{1}{4}
\end{array}\right)
\qquad
\left(\begin{array}{rrrrr}
\frac{1}{3} & \frac{4}{9} & \frac{2}{3} & 0 & 0 \\
\frac{1}{2} & -\frac{1}{12} & -\frac{1}{8} & \frac{3}{8} & \frac{9}{16} \\
0 & \frac{1}{2} & -\frac{1}{4} & \frac{3}{4} & -\frac{3}{8} \\
1 & -\frac{1}{6} & -\frac{1}{4} & -\frac{1}{4} & -\frac{3}{8} \\
0 & 1 & -\frac{1}{2} & -\frac{1}{2} & \frac{1}{4}
\end{array}\right)
\end{equation*}

The promotion matrix and its inverse are
\begin{equation*}
\left(\begin{array}{rrrrr}
0 & 0 & 0 & 1 & 0 \\
0 & 0 & 0 & 0 & 1 \\
1 & 0 & 0 & 0 & 0 \\
0 & 0 & 1 & 0 & 0 \\
0 & 1 & 0 & 0 & 0
\end{array}\right)
\qquad
\left(\begin{array}{rrrrr}
0 & 0 & 1 & 0 & 0 \\
0 & 0 & 0 & 0 & 1 \\
0 & 0 & 0 & 1 & 0 \\
1 & 0 & 0 & 0 & 0 \\
0 & 1 & 0 & 0 & 0
\end{array}\right)
\end{equation*}

The matrix which interpolates between promotion and rotation and its inverse are
\begin{equation*}
\left(\begin{array}{rrrrr}
\frac{1}{[3]} & \frac{[4]}{[3]^2} & 0 & \frac{[4]}{[2][3]} & 0 \\
\frac{1}{[2]} & \frac{-1}{[3][2]^2} & \frac{[3]}{[2]^3} & \frac{-1}{[2]^3} & \frac{[3]^2}{[2]^4} \\
1 & \frac{-1}{[2][3]} & \frac{-1}{[2]^2} & \frac{-1}{[2]^2} & \frac{-[3]}{[2]^3} \\
0 & \frac{1}{[2]} & \frac{[3]}{[2]^2} & \frac{-1}{[2]^2} & \frac{-[3]}{[2]^3} \\
0 & 1 & \frac{-1}{[2]} & \frac{-1}{[2]} & \frac{1}{[2]^2}
\end{array}\right)
\qquad
\left(\begin{array}{rrrrr}
\frac{1}{[3]} & \frac{[4]}{[3]^2} & \frac{[2]}{[3]} & 0 & 0 \\
\frac{1}{[2]} & -\frac{1}{[2]^2[3]} & -\frac{1}{[2]^3} & \frac{[3]}{[2]^3} & \frac{[3]^2}{[2]^4} \\
0 & \frac{1}{[2]} & -\frac{1}{[2]^2} & \frac{[3]}{[2]^2} & -\frac{[3]}{[2]^3} \\
1 & -\frac{1}{[2][3]} & -\frac{1}{[2]^2} & -\frac{1}{[2]^2} & -\frac{[3]}{[2]^3} \\
0 & 1 & -\frac{1}{[2]} & -\frac{1}{[2]} & \frac{1}{[2]^2}
\end{array}\right)
\end{equation*}

In this example the matrix which intertwines promotion and rotation can be given explicitly.
This matrix and its inverse are:
\begin{equation*}
\left(\begin{array}{rrrrr}
1 & -\frac{[2]}{[3]} & 0 & 0 & -\frac 1{[2]} \\
0 & 1 & -\frac 1{[2]} & -\frac 1{[2]} & \frac 1{[2]^2} \\ 
0 & 0 & 1 & 0 & -\frac 1{[2]} \\
0 & 0 & 0 & 1 & -\frac 1{[2]} \\
0 & 0 & 0 & 0 & 1
\end{array}\right)
\qquad
\left(\begin{array}{rrrrr}
1 & \frac{[2]}{[3]} & \frac 1{[3]} & \frac 1{[3]} & \frac {[2]}{[3]} \\
0 & 1 & \frac 1{[2]} & \frac 1{[2]} & \frac 1{[2]^2} \\
0 & 0 & 1 & 0 & \frac 1{[2]} \\
0 & 0 & 0 & 1 & \frac 1{[2]} \\
0 & 0 & 0 & 0 & 1
\end{array}\right)
\end{equation*}
\end{ex}

Our construction of the interpolating matrix uses the cactus group $\fC_r$.
This group acts on the set of standard tableaux of shape $\lambda$. This
action is given explicitly in \cite{Kirillov1995} (where the cactus groups are
implicit) and implicitly in \cite{Henriques2006d} (where the cactus groups are
explicit). The relationship between these two papers was established in
\cite{Chmutov2016}.

The Hecke algebra, $H_r$, is a $q$-analogue of the symmetric group algebra $\bQ\,\fS_r$. As a $\bQ(q)$-algebra it is split semisimple and each 
$U_\lambda$ has a $q$-analogue which is also absolutely irreducible.
Our main technical tool is that, for $r>1$, there is a homomorphism
$\fC_r\to H_r$. This homomorphism is constructed implictly using 
quantum Schur-Weyl duality and the result that the category of type I
finite dimensional representations of a quantised enveloping algebra is a
coboundary category; this is part of the construction in \cite{Drinfelprimed1989}.
This implies that, for each $\lambda\vdash r$, there
is an action of $\fC_r$ on $U_\lambda$.
This action was made explicit in \cite{Westbury2018} in that the matrices
representing a set of generators with respect to the seminormal basis
of $U_\lambda$ are given explicitly. Here we make use of these matrices.

The contents of the sections are:
\begin{description}
	\item[Cyclic sieving phenomenon] In this section we give the background
on the cyclic sieving phenomenon and deduce the cyclic sieving phenomenon
from Theorem~\ref{thm:main}.
\item[Cactus groups] In this section we define the cactus groups by
finite presentations and give the results that are needed in the proof of
Theorem~\ref{thm:main}.
\item[Hecke algebras] In this section we recall the construction of
the representations of the Hecke algebra in the seminormal basis and
prove Theorem~\ref{thm:main}.
\item[Conclusion] In this section we put the statement and method of proof
of Theorem~\ref{thm:main} in the context of the representation theory
of quantised enveloping algebras and discuss the main difficulty in
generalising Theorem~\ref{thm:main} to this context.
\end{description}

\section{Cyclic sieving phenomenon}
In this section we explain how Theorem~\ref{thm:main} gives the
cyclic sieving phenomenon in \cite{Rhoades2010}.

Let $r>1$ and $\omega$ be a primitive $r$-th root of unity. Recall the
definition of the cyclic sieving phenomenon from \cite{Reiner2004}.
\begin{defn}
Let $X$ be a finite set and $c\colon X\to X$ a bijection that satisfies
$c^r=1$. If the polynomial $P\in \bZ[q]$ \footnote{Note that the $q$ in this section is not the $q$ used in other sections.} satisfies
\begin{equation*}
	P(\omega^k) = |\Fix(c^k)|
\end{equation*}
for all $k$ where $\Fix$ is the set of fixed points then the triple $(X,c,P)$
exhibits the \Dfn{cyclic sieving phenomenon}.
\end{defn}
Let $\lambda\vdash r$ be a rectangular shape and let $X$ be the set of
standard tableaux of shape $\lambda$. Denote the promotion operator
on $X$ by $p$. The problem is to determine a polynomial $P$ such that
$(X,p,P)$ exhibits the cyclic sieving phenomenon.

The linear version of the cyclic sieving phenomenon is:
\begin{defn}
Let $U$ be a finite dimensional vector space over $\bQ$ and $\rho\colon U\to U$ an isomorphism that satisfies $\rho^r=1$. A \Dfn{character polynomial} is a 
polynomial $P\in \bQ[q]$ that satisfies
\begin{equation*}
	P(\omega^k) = \tr(\rho^k)
\end{equation*}
for all $k$.
\end{defn}

The two basic properties of the character polynomial are:
\begin{itemize}
\item If $\rho$ is the permutation matrix of a bijection $c$ then $(X,c,P)$
satisfies the cyclic sieving phenomenon if and only if $P$ is a
character polynomial of $\rho$.

\item Given $\rho\colon U\to U$ with character polynomial $P$ and
$\rho'\colon U'\to U'$ with character polynomial $P'$ then 
$(U,\rho)$ and $(U',\rho')$ are isomorphic if and only if $P=P'$.
\end{itemize}
It then follows from Theorem~\ref{thm:main} that $(X,p,P)$ exhibits the cyclic sieving phenomenon if and only if $P$ is a character polynomial for the
action of the long cycle on $U_\lambda$.

This character polynomial can be determined using:
\begin{prop}\label{prop:fr} Let $U$ be a representation of $\fS_r$. Then the principal specialisation of the Frobenius character of $U$ is a character polynomial for the action of the long cycle.
\end{prop}

\begin{lemma} Let $\lambda\vdash r$. Take $P$ to be the $q$-analogue of the hook-length formula
	\begin{equation*}
		\frac{[n]!}{\prod_{(i,j)\in \lambda} [h(i,j)]}
	\end{equation*}
	where $h(i,j)$ is the hook length of the cell $(i,j)$ and $[n]$ is given by
	\begin{equation*}
		[n] = \frac{1-q^n}{1-q}
	\end{equation*}
Then $P$ is a character polynomial for the action of the long cycle on $U_\lambda$.
\end{lemma}

\begin{proof} This is an application of Proposition~\ref{prop:fr}.
The Frobenius character of $U_\lambda$ is the Schur function $s_\lambda$;
and the principal specialisation of $s_\lambda$ is given by
the $q$-analogue of the hook-length formula.
\end{proof}

Other interpretations of this polynomial are available. For example,
this polynomial is the generating function for the statistic major index
on the set of standard tableaux of shape $\lambda$.

The conclusion is that if $P$ is the $q$-analogue of the hook-length formula
then $(X,p,P)$ exhibits the cyclic sieving phenomenon. This is the main
theorem of \cite{Rhoades2010}.

\section{Cactus group}\label{sec:cactus}
The finite presentations of the cactus groups are:
\begin{defn}\label{defn:cactus} The \Dfn{$r$-fruit cactus group}, $\fC_r$, has generators $s_{p,\,q}$ for $1\le p<q\le r$ and defining relations
\begin{itemize}
	\item $s_{p,\,q}^2=1$
	\item $s_{p,\,q}\, s_{k,\,l}=s_{k,\,l}\, s_{p,\,q}$ if $\os{p,q}\cap\os{k,l}=\emptyset$
	\item $s_{p,\,q}\, s_{k,\,l}=s_{p+q-l,\,p+q-k}\, s_{p,\,q}$ if $\os{k,l}\subseteq\os{p,q}$
\end{itemize}
\end{defn}

Let $\fS_r$ be the symmetric group on $r$ letters.
There is a homomorphism $\fC_r\rightarrow \fS_r$ defined by $s_{p,\,q}\mapsto \widehat{s}_{p,\,q}$ where $\widehat{s}_{p,\,q}$ is the permutation
\begin{equation}\label{eq:hom}
\widehat{s}_{p,\,q}(i)=\begin{cases}
p+q-i & \text{if $p\le i\le q$} \\
i & \text{otherwise}
\end{cases}
\end{equation}

Note that $\fC_r$ is generated by $s_{1,q}$ for $2\le q\le r$, since
\begin{equation}\label{eqn:gen}
s_{p,\,q}=s_{1,\,q}s_{1,\,q-p}s_{1,\,q}
\end{equation}

Next we introduce several sets of generators for the cactus groups,
following \cite{Kirillov2001}.

The first set is $q_i=s_{1,\,i}$ for $1\leqslant i\leqslant n$.
These are generators, as noted in \eqref{eqn:gen}.
The second set of generators
are $p_i$ given by the relations
\begin{equation*}
p_i=\begin{cases}
q_1 & \text{if $i=1$}\\
q_{i-1}\,q_i & \text{if $i>1$}
\end{cases} \qquad
q_i = p_1\,p_2\dotsb p_i
\end{equation*}
The third set of generators are $t_i$ given by
\begin{equation*}
t_i=\begin{cases}
p_1 & \text{if $i=1$}\\
p_i\,p_{i-1}^{-1} & \text{if $i>1$}
\end{cases} \qquad
p_i = t_i\,t_{i-1}\dotsb t_1
\end{equation*}

Then we also have
\begin{equation*}
t_i=\begin{cases}
q_1 & \text{if $i=1$}\\
q_1\,q_2\,q_1 & \text{if $i=2$}\\
q_{i-1}\,q_i\,q_{i-1}\,q_{i-2} & \text{if $i>2$}
\end{cases}
\end{equation*}

The images of these three sets of generators under the homomorphism to $\fS_r$ are:
\begin{itemize}
\item the image of $t_i$ is the transposition $(i,i+1)$
\item the image of $p_i$ is the cycle $(i,1,2,\dotsc,i-1)$
\item the image of $q_i$ is the involution
\begin{equation*}
q_i(j)=\begin{cases}
i-j+1  & \text{if $1\leqslant j\leqslant i$}\\
j & \text{if $j>i$}
\end{cases}
\end{equation*}
In particular, the image of $p_{r-1}\in\fS_r$ is the long cycle and
the image of $q_r\in\fS_r$ is the longest element.
\end{itemize}

There is a dual version of these generators.
\begin{align*}
	v_i &= t_i\,t_{i+1}\dotsc t_{r-1}\\
	w_i &= v_{r-1}\,v_{r-2}\dotsc v_{r-i}
\end{align*}

The following is \cite[Proposition~1.4]{Kirillov1995} and is also clear from
the growth diagram.
\begin{lemma}\label{lem:cyclic} For $r>1$,
	\begin{equation*}
	p_{r-1}^r=w_{r-1}\,q_{r-1}
	\end{equation*}
\end{lemma}

For a standard tableau, $T$, let $s_iT$ be the
tableau obtained from $T$ by interchanging $i$ and $i+1$.
Define an involution $t_i$ on standard tableaux by
\begin{equation*}
	t_i(T) = \begin{cases}
		s_iT & \text{if $s_iT$ is standard} \\
		T & \text{otherwise}
	\end{cases}
\end{equation*}

Then these involutions generate an action of the cactus group $\fC_r$
on standard tableaux of size $r$. The action of the element $p_{r-1}\in\fC_r$
is the same as jeu-de-taquin promotion.

\begin{prop}\label{prop:order} The operator $p_{r-1}$ acting on rectangular tableaux of size $r$
	satisfies
	$$p_{r-1}^r=1$$
\end{prop}

\begin{proof}
	By Lemma~\ref{lem:cyclic} it is sufficient to show $w_{n-1} = q_{n-1}$.
	This follows from two observations on the reverse-complement.
	Recall that reverse-complement is the involution given by rotating a rectangular shape
	standard tableaux through a half-turn and reversing the numbering.
	
	The first observation is that conjugating by reverse-complement interchanges the two actions of the cactus groups. In particular, $w_{r-1}$ is the conjugate of $q_{n-1}$.
	The second observation is that $q_{r-1}$ commutes with reverse-complement.
\end{proof}

\section{Hecke algebras}\label{sec:hecke}
Let $\bQ(q)$ be the field of rational functions in an indeterminate $q$.
The quantum integers $[n]\in \bQ(q)$ are defined by
\begin{equation*}
[n] = \frac{q^n-q^{-n}}{q-q^{-1}}
\end{equation*}


\begin{defn}
	The Hecke algebra $H_r(q)$ is generated by $u_i$ for $1\le i\le r-1$
	and the defining relations are
	\begin{align*}
	u_i^2 &= -[2]\, u_i \\
	u_i\,u_{i+1}\,u_i-u_i&=u_{i+1}\,u_i\,u_{i+1}-u_{i+1} \\
	u_i\,u_j&=u_j\,u_i\qquad\text{for $|i-j|>1$}
	\end{align*}
\end{defn}

Let $\sigma_i$ be the standard generators of the braid group, $B_r$.
Then we have a homomorphism $B_r\to H_r(q)$ given by
\begin{equation*}
	\sigma_i^{\pm 1} \mapsto q^{\pm1}  + u_i
\end{equation*}
The image of $\sigma_i$ then satisfies
\begin{equation*}
	\sigma_i-\sigma_i^{-1} = q - q^{-1}
\end{equation*}
Young's seminormal forms are representations of the symmetric groups.
These were introduced in \cite[Theorem~IV]{Young1932}. Here we give
the analogous construction for the Hecke algebras.

Fix a shape $\lambda$ of size $r$. Then we construct a 
representation of $H_r(q)$ on the vector space with basis the set
of standard tableaux of shape $\lambda$.

The \Dfn{content vector} of a standard tableau $T$ of size $n$ is a function
$\ct_T\colon [1,2,\dotsc ,n]\to\bZ$. The entry
$\ct_T(k)\in\bZ$ is given by $\ct_T(k)=j-i$ if $k$ is in box $(i,j)$ in $T$.
The content vector of $T$ determines $T$. The \Dfn{axial distance}
is $a_T(i) = \ct_{T}(i+1)-\ct_{T}(i)$.

Define a linear operator $\overline{s}_i$ by
\begin{equation*}
	\overline{s}_i(T) = \begin{cases}
		s_iT & \text{if $s_iT$ is standard} \\
		0 & \text{otherwise}
	\end{cases}
\end{equation*}

The following is \cite[Theorem~3.22]{Ram2003}. The seminormal representations of
$H_r(q)$ are defined by giving the matrices representing the generators.
\begin{prop}
	The action of $u_i$ is given by
	\begin{equation*} u_i\:T = 
	\begin{cases}
	-\frac{[a-1]}{[a]}\:T + \frac{[a-1]}{[a]}\:\overline{s}_iT &\text{if $\ct_T(i+1)>\ct_T(i)$}\\
	-\frac{[a-1]}{[a]}\:T + \frac{[a+1]}{[a]}\:\overline{s}_iT &\text{if $\ct_T(i)>\ct_T(i+1)$}	
	\end{cases}
	\end{equation*}
where $a$ is the axial distance $a_T(i)$.
\end{prop}

Assume $T$ and $s_iT$ are standard and $a_T(i)>1$. Then on the subspace
with ordered basis $(T,s_iT)$.
\begin{equation*}
u_i =	\left[ \begin {array}{cc} -\frac {[a-1]}{ [a]  }& \frac{[a-1]}{[a]}
\\ \noalign{\medskip}{\frac {[a+1]  }{
		[a]}}&-\frac {[a+1]}{ [a]  }\end {array} \right]
	\qquad
\sigma_i=\left[ \begin {array}{cc} \frac {q^a}{ [a]  }&\frac{[a-1]}{[a]}\:
\\ \noalign{\medskip}{\frac {[a+1]  }{
		[a]}}&\frac {-q^{-a}}{ [a]  }\end {array} \right]
\end{equation*}

The action of $\fC_r$ is defined by giving the matrices representing the generators
$\{t_i\}$. The following is \cite[Theorem~4.8]{Westbury2018}.
\begin{thm} The action of $t^{(q)}_i$ is given by
	\begin{equation*} t^{(q)}_i (T) = 
	\begin{cases}
	-\frac{1}{[a]}\:T + \frac{[a-1]}{[a]}\:s_iT &\text{if $\ct_T(i+1)>\ct_T(i)$}\\
	-\frac{1}{[a]}\:T + \frac{[a+1]}{[a]}\:s_iT &\text{if $\ct_T(i)>\ct_T(i+1)$}	
	\end{cases}
	\end{equation*}
\end{thm}

Assume $T$ and $s_iT$ are standard and $a_T(i)>1$. Then on the subspace
with ordered basis $(T,s_iT)$.
\begin{equation*}
t^{(q)}_i=\left[ \begin {array}{cc} \frac {1}{ [a]  }&\frac{[a-1]}{[a]}
\\ \noalign{\medskip}{\frac {[a+1]  }{
		[a]}}&-\frac {1}{ [a]  }\end {array} \right]
\end{equation*}

Let $p^{(q)}_{r-1}$ be the matrix representing $p_{r-1}\in\fC_r$. Then we prove
Theorem~\ref{thm:main} by showing that a modification of $p^{(q)}_{r-1}$
is an interpolating matrix.

\begin{prop} The operator $p^{(q)}_{r-1}$ acting on rectangular tableaux of size $r$
	satisfies
	$$(p^{(q)}_{r-1})^r=1$$
\end{prop}

\begin{proof} The proof is the same as the proof of Proposition~\ref{prop:order}.
\end{proof}

\begin{lemma} The evaluation of $p^{(q)}_{r-1}$ at $q=1$ is defined and gives the action of the long cycle. 
\end{lemma}

\begin{proof} The action of $\fC_r$ on $U_\lambda$ can be evaluated at $q=1$.
This action factors through the homomorphism in \eqref{eq:hom} to give the
standard action of $\fS_r$ in the seminormal basis. In particular,
the homomorphism in \eqref{eq:hom} maps $p^{(q)}_{r-1}$ to the long cycle in
$\fS_r$.
\end{proof}

The matrix $t^{(q)}_i$ is not regular at $q=0$ as $\frac{[a+1]}{[a]}$ has a
simple pole at $q=0$. Let $D$ be the diagonal matrix whose diagonal entry
corresponding to $T$ is $q^{\mathrm{inv}(T)}$ where $\mathrm{inv}(T)$ is the
inversion number of $T$. Let $\widehat{t}^{(q)}_i$ be $t^{(q)}_i$ conjugated by $D$.

\begin{lemma}
The matrix $\widehat{t}^{(q)}_i$ is regular at $q=0$ and the evaluation at $q=0$ is the matrix of the involution $t_i$.
\end{lemma}
\begin{proof} By inspection.
\end{proof}

For any $\lambda\vdash r$, the matrices $\widehat{t}^{(q)}_i$ generate an action of $\fC_r$ and the matrix
of any element is regular at $q=0$. Let $\widehat{p}^{(q)}_{r-1}$ be the matrix
of $p_{r-1}\in\fC_r$. Then we have shown that if $\lambda$ has rectangular shape then $\widehat{p}^{(q)}_{r-1}$
is an interpolating matrix and so have proved Theorem~\ref{thm:main}.

\section{Conclusion}
In conclusion we put this result in the context of the representation theory
of quantum groups and discuss the generalisation of the statement and method of proof of Theorem~\ref{thm:main} in this context.

Let $\sC$ be a finite type Cartan matrix and $U_q(\sC)$ the associated
quantised enveloping algebra. Let $\varpi$ be a dominant weight in the
weight lattice of $\sC$. Let $V_q(\varpi)$ be the type I highest weight
representation of $U_q(\sC)$ and let $C(\varpi)$ be the crystal of $V_q(\varpi)$.

The cactus group $\fC_r$ acts on the crystal $\otimes^r C(\varpi)$.
This implies that $\fC_r$ acts on the set of highest weight words
in $\otimes^r C(\varpi)$, preserving the weight. 
In particular $\fC_r$ acts on $\Inv_r(C(\varpi))$, the set of highest weight
words of weight 0. Let $c_0$ be the permutation matrix of the action of
$p_{r-1}\in\fC_r$.

Let $L(\sC)$ be the semisimple Lie algebra of $\sC$. Let $V(\varpi)$ be the
highest weight representation of $L(\sC)$. The symmetric group $\fS_r$
acts on the representation $\otimes^r V(\varpi)$. 
This implies that $\fS_r$ acts on the space of highest weight tensors
in $\otimes^r V(\varpi)$, preserving the weight. 
In particular $\fS_r$ acts on $\Inv_r(V(\varpi))$, the space of invariant tensors.
Let $c_1$ be the matrix representing the long cycle with respect to a
chosen basis.

The generalisation of Theorem~\ref{thm:main} is that the matrices
$c_0$ and $c_1$ are conjugate. Theorem~\ref{thm:main} is the case $\sC$
has type $A$ and $\varpi$ is the first fundamental weight; so 
$L(\sC)$ is $\fsl(n)$ for some $n$ and $V(\varpi)$ is the vector representation.

Now we attempt to generalise the proof of Theorem~\ref{thm:main}.
The cactus group $\fC_r$ acts on the representation $\otimes^r V_q(\varpi)$.
This implies that $\fC_r$ acts on the set of highest weight tensors
in $\otimes^r V_q(\varpi)$, preserving the weight. 
In particular $\fC_r$ acts on $\Inv_r(V_q(\varpi))$, the space of invariant tensors. Let $c_q$ be the matrix representing $p_{r-1}\in\fC_r$. with respect to a
chosen basis.

Then $c_q$ has the following two of the three properties of an interpolating
matrix. These properties are independent of the choice of basis of
$\Inv_r(V_q(\varpi))$.
\begin{itemize}
	\item If the evaluation of $c_q$ at $q=1$ is defined then this gives the action of the long cycle.
	\item $c_q^r=1$.
\end{itemize}

Hence to show that $c_q$ is an interpolating matrix it remains to show
that there exists a basis of $\Inv_r(V_q(\varpi))$ such that $c_q$ is
regular at $q=0$ and the evaluation at $q=0$ is $c_0$. If
$V(\varpi)$ admits an invariant symplectic form then there are sign issues.
These sign issues can be dealt with by taking $V(\varpi)$ to be an odd super vector
space so that the symplectic form becomes a symmetric inner product.
This is proved in \cite{Westbury2016} using the theory of based modules
in \cite{Lusztig1993}.
It would be preferable to have a self-contained proof.

In the proof of Theorem~\ref{thm:main} we showed that a modification of the seminormal basis has the desired property. The seminormal basis
of $\Inv_r(V_q(\varpi))$ is defined for any $\varpi$ such that all nonzero
weight spaces of $V(\varpi)$ are one dimensional. Examples are,
\begin{itemize}
	\item All minuscule representations.
	\item The symmetric powers of the vector representation of $\fsl(n)$.
	\item The fundamental representation of $G_2$.
\end{itemize}
This raises the question of whether a modification of the seminormal basis
has the desired property.

\bibliography{All}{}
\bibliographystyle{plain}
\end{document}